\documentclass[a4paper,11pt,reqno]{amsart}
\usepackage{amsmath}
\usepackage{amssymb}
\usepackage{xcolor}
\usepackage{enumerate}

\usepackage{mathrsfs}

\theoremstyle{plain} 
\newtheorem{theorem}{Theorem}[section]
\newtheorem{proposition}[theorem]{Proposition}
\newtheorem{lemma}[theorem]{Lemma}
\newtheorem{corollary}[theorem]{Corollary}
\newtheorem{remark}[theorem]{Remark}

\makeatletter
    
    \@addtoreset{equation}{section}
  \makeatother

\title[Elephant random walk with a power law memory]{Phase transitions for a unidirectional elephant random walk with a power law memory} 
\thanks{R.R. thanks Keio University for its hospitality during multiple visits. M.T. and H.T. thank the Indian Statistical Institute for its hospitality. M.T. is partially supported by JSPS KAKENHI Grant Numbers JP19H01793, JP19K03514 and JP22K03333. H.T. is partially supported by JSPS KAKENHI Grant Number JP19K03514. 
}
\author{Rahul Roy}
\address{Indian Statistical Institute, New Delhi, India}
\email{rahul@isid.ac.in}
\author{Masato Takei}
\address{Department of Applied Mathematics, Faculty of Engineering, Yokohama National University, Yokohama, Japan}
\email{takei-masato-fx@ynu.ac.jp}
\author{Hideki Tanemura}
\address{Department of Mathematics, Keio University, Yokohama, Japan}
\email{tanemura@math.keio.ac.jp}

\begin{document}

\begin{abstract}
For the standard elephant random walk, Laulin (2022) studied the case when the increment of the random walk is not uniformly distributed over the past history instead has a power law distribution.
We study such a problem for the unidirectional elephant random walk introduced by Harbola, Kumar and Lindenberg (2014).
Depending on the memory parameter $p$ and the power law exponent $\beta$, we obtain three distinct phases in one such phase the elephant travels only a finite distance almost surely, and the other two phases are distinguished by the speed at which the elephant travels. 
\end{abstract}

\maketitle

\section{Introduction}
\label{sec:intro}
The elephant random walk (ERW),  introduced by Sch\"{u}tz and Trimper \cite{SchutzTrimper04}, is a one-dimensional discrete time random walk whose incremental steps are $\pm 1$. Unlike the simple random walk, the ERW keeps a track of each of the incremental steps taken throughout its history. Formally, let $S_n$ be the location of the ERW at time $n$. We have $S_0 \equiv 0$ and, for $s \in [0,1]$,
\begin{align}
\label{r-1st step}
S_1 &= X_1 := \begin{cases}
+1 & \text{ with probability }s\\
-1 &  \text{ with probability }1-s.
\end{cases}
\end{align}
Subsequently, for $n \geq 1$, 
$$
S_{n+1} = S_n + X_{n+1}
$$
where, for  $p \in (0,1)$ and $U_n$, a uniform random variable on $\{1, \ldots , n\}$,  
\begin{align}
\label{r-increments}
X_{n+1} &= \begin{cases}
+X_{U_n} & \text{ with probability }p\\
-X_{U_n} &  \text{ with probability }1-p.
\end{cases}
\end{align}
Here we assume that $\{U_k: k \in \mathbb{N} \}$ is a collection of independent random variables.

Depending on the {\it memory parameter}\/ $p$, we observe three distinct phases,
\begin{itemize}
\item [(i)] when $0 < p < 3/4$, the ERW exhibits a diffusive phase where 
$$
\frac{S_n}{\sqrt{n}} \overset{\text{d}}{\to} N\left(0, \frac{1}{3-4p}\right) \qquad \text{Coletti {\it et al.}\/ \cite{Collettietal17a},}
$$
\item[(ii)] when $p= 3/4$, the ERW exhibits a critical phase where  
$$
\frac{S_n}{\sqrt{n\log n}} \overset{\text{d}}{\to} N(0, 1)  \qquad \text{Coletti {\it et al.}\/ \cite{Collettietal17a},}
$$
\item[(iii)] when $3/4 < p < 1$, the ERW exhibits a superdiffusive phase where, for a finite, non-degenerate, non-Gaussian random variable $L$,
$$
\lim_{n \to \infty} \dfrac{S_n}{n^{2p-1}} = L\quad \mbox{a.s. and in $L^2$} \qquad \text{Bercu \cite{Bercu18},}
$$
and
$$
\frac{S_n - Ln^{2p-1}}{\sqrt{n}} \overset{\text{d}}{\to} N\left(0,  \frac{1}{4-3p}\right) \qquad \text{Kubota and Takei \cite{KubotaTakei19JSP}.}
$$
\end{itemize}
It should be noted here that, like the simple random walk, the ERW is also Markovian, albeit inhomogeneous, because
$$
S_{n+1} = \begin{cases}
S_n +1 & \text{ with probability }\frac{1}{2} + (p - \frac{1}{2})\frac{S_n}{n}\\
S_n -1&  \text{ with probability }\frac{1}{2} - (p - \frac{1}{2})\frac{S_n}{n}.
\end{cases}
$$

Other properties of the ERW as well as its multidimensional version have been studied, see the excellent thesis of Laulin \cite{Laulin22PhD} and the references therein. More recently Bertoin \cite{Bertoin22Counting}, Gu\'{e}rin {\it et al.} \cite{GuerinLaulinRaschel23} and  Qin \cite{Qin23} have studied the recurrence and transience properties of the ERW and its higher dimensional analogues.


Gut and Stadm\"{u}ller  \cite{GutStadtmuller21JAP} considered the following two models. Let $\{U_k : k \in \mathbb{N}\}$ be independent uniform random variables as earlier. Fix $m \in \mathbb{N}$.
\begin{itemize}
\item Remembering the distant past of size $m$: Let $\{U_{m,n} : n \geq m\}$ be i.i.d. random variables, independent of $\{U_k : k \in \mathbb{N}\}$ with $U_{m,n} \overset{\text{d}} {=} U_m$ for $n  \geq m$. The ERW is described by the increments (\ref{r-increments}) for $n < m$, and, for $n \geq m$,  with 
\begin{equation}
\label{r-GSinitial}
X_{n+1} = \begin{cases}
+X_{U_{m,n}} & \text{ with probability }p\\
-X_{U_{m,n}} &  \text{ with probability }1-p.
\end{cases}
\end{equation}
\item Remembering the recent past of size $m$: Let $\{U^m_{n} : n \geq m\}$ be a collection of independent random variables, independent of  $\{U_k : k \in \mathbb{N}\}$ with 
$U_n^m$ having a uniform distribution over $\{n-m+1, \ldots, n\}$. The ERW is described by the increments (\ref{r-increments}) for $n < m$,  and, for $n \geq m$, with increments 
\begin{align}
\label{r-GSinitial2}
X_{n+1} &= \begin{cases}
+X_{U^m_n} & \text{ with probability }p\\
-X_{U^m_n} &  \text{ with probability }1-p.
\end{cases}
\end{align}
\end{itemize} 
\begin{remark}
$U_1 \equiv 1$, so for $m=1$, the increments given by \eqref{r-GSinitial} can be seen to be i.i.d. For $m \geq 2$, the process which remembers the distant past exhibits a superdiffusive behaviour.
However the process which remembers the recent past exhibits a diffusive behaviour for $m \geq 1$.
\end{remark}

Laulin \cite{Laulin22ECP} continued the study further.
Let $\beta > 0$ and $\{\beta_{n+1}: n \in \mathbb{N} \}$ be independent random variables with 
\begin{align}
\label{r-beta1}
P(\beta_{n+1} = k) = \begin{cases} \dfrac{\beta + 1}{n} \cdot \dfrac{\mu_k}{\mu_{n+1}} & \text { for } 1 \leq k \leq n\\
0  & \text{ otherwise}
\end{cases}
\end{align}
where
\begin{equation}
\label{r-mu1}
\mu_n = \frac{\Gamma(n+ \beta)}{\Gamma(n)\Gamma( \beta +1)} \sim \frac{n^\beta}{\Gamma(\beta +1)} \quad \mbox{as $n \to \infty$.}
\end{equation}
The increments are given by
\begin{align}
\label{r-beta1st step}
X_1 := \begin{cases}
+1 & \text{ with probability }s\\
-1 &  \text{ with probability }1-s,
\end{cases}
\end{align}
for some $s \in [0,1]$, and, for  $p \in (0,1)$, $n \geq 1$, 
\begin{align}
\label{r-betaincrements}
X_{n+1} &= \begin{cases}
+X_{\beta_{n+1}} & \text{ with probability }p\\
-X_{\beta_{n+1}} &  \text{ with probability }1-p.
\end{cases}
\end{align}
Laulin \cite{Laulin22ECP} showed that the ERW obtained from the increments described above exhibits a phase transition in the sense that, it is
$$
\text{diffusive for } p < \frac{4\beta + 3}{4(\beta + 1)} \quad \text{ and superdiffusive for } p >\frac{4\beta + 3}{4(\beta + 1)}.
$$
See Chen and Laulin \cite{ChenLaulin23} for higher-dimensional analogues.


Harbola {\it et al.} \cite{HarbolaEtal14PRE} introduced a unidirectional ERW where the first step is 
\begin{align}
\label{r-1st step2}
S_1 &= X_1 := \begin{cases}
+1 & \text{ with probability }s\\
0 &  \text{ with probability }1-s,
\end{cases}
\end{align}
and subsequently, with the collection $\{U_k: k \in \mathbb{N} \}$ as earlier, 
\begin{align}
\label{r-lazy1}
\begin{split}
\text{if } X_{U_n} = 1 \text{ then } & X_{n+1} := \begin{cases}
+1 & \text{ with probability }p \\
0 &  \text{ with probability }1-p,
\end{cases}\\
\text{if } X_{U_n} = 0 \text{ then } & X_{n+1} := \begin{cases}
+1 & \text{ with probability } q\\
0 &  \text{ with probability }1-q,
\end{cases}
\end{split}
\end{align}
here $s \in [0,1]$, $p \in (0,1)$ and $q \in [0,1]$ are the parameters of the process.
Harbola {\it et al.} \cite{HarbolaEtal14PRE} showed that, for the  random walk $\Xi _n := \sum_{k=1}^n X_k$, with $\{X_k : k \in \mathbb{N}\}$ as above,
$$
E[\Xi_n] \sim \begin{cases} \dfrac{qn}{1-q} & \text{ if } q > 0\\
\dfrac{sn^p}{\Gamma(1+p)} & \text{ if } q = 0,
\end{cases}
$$
and Coletti {\it et al.} \cite{CollettiGavadeLima19JSM} observed that for $q>0$, there are three distinct phases depending on $p-q$.

When $q=0$ and $s=1$ the walk is the `laziest elephant random walk (LERW)', which is thus given by $\Xi_0=0$, 
\begin{align}
\label{r-lazydef}
 \Xi _n := \sum_{k=1}^n X_k \text{ with } 
 X_1 \equiv 1,\,X_{n+1} = \begin{cases}
X_{U_n} & \text{with probability }p\\
0 &  \text{with probability }1-p.
\end{cases}
\end{align}
Miyazaki and Takei \cite{MiyazakiTakei20JSP}  studied this model and showed that
$$
\frac{\Xi_n}{n^p} \overset{\text{d}}{\to}  {\mathcal W}, \text{ where $\mathcal W$ has a Mittag--Leffler distribution with parameter $p$,}
$$ 
and they also proved that
$$
\frac{\Xi_n - {\mathcal W}n^p}{\sqrt{{\mathcal W}n^p}} \overset{\text{d}}{\to} N(0,1).
$$


Our study here is similar to Laulin \cite{Laulin22ECP}, however for the LERW given in (\ref{r-lazydef}) and having a memory distribution  (\ref{r-beta1}) with $\beta >-1$.
In particular, our model is given by $\Xi_0=0$,
\begin{align}
\label{r-lazydef2}
\Xi _n := \sum_{k=1}^n X_k \text{ with } 
 X_1 \equiv 1,\, X_{n+1} = \begin{cases}
X_{\beta_{n+1}} & \text{with probability }p\\
0 &  \text{with probability }1-p,
\end{cases}
\end{align}
where $\{\beta_{n+1} : n \in \mathbb{N} \}$ are as given in (\ref{r-beta1}) with $\beta > -1$.

For this model we obtain three distinct rates of growth of $\Xi_n$ depending on the parameter $\beta$. We employ a martingale method and a coupling with a multi-type branching process. The latter is of intrinsic interest because it may be applied to other ERW models. In the next section we state our results and in subsequent sections we 
present the proofs of our results.

\section{Results}
\label{sec:results}

Throughout the rest of the article we restrict ourselves to the model given by \eqref{r-lazydef2}.
Let $\mathcal{F}_0$ be the trivial $\sigma$-field, and $\mathcal{F}_n$ be the $\sigma$-field generated by $X_1,\ldots,X_n$. By \eqref{r-beta1}, we have
\begin{align}
E[X_{n+1} \mid \mathcal{F}_n] &=  p \cdot  E[ X_{\beta_{n+1}} \mid \mathcal{F}_n] \notag \\
&= p \cdot \sum_{k=1}^n X_k P(\beta_{n+1}=k) 
= \dfrac{p(\beta+1)}{n\mu_{n+1}} \sum_{k=1}^n X_k \mu_k. \label{eq:CondEX_n+1}
\end{align}
We introduce an auxiliary process $\{\Sigma_n: n \in \mathbb{N}\}$ 
defined as 
\begin{align*}
\Sigma_n := \sum_{k=1}^n X_k \mu_k \quad \mbox{for $n \in \mathbb{N}$.}
\end{align*} 
Note that $\Sigma_1 = X_1 \mu_1=1$ a.s. Since
\begin{align}
\label{X_n+1}
E[X_{n+1} \mid \mathcal{F}_n] &= \dfrac{p(\beta+1)}{n\mu_{n+1}} \cdot \Sigma_n,
\end{align}
we have
\begin{align*}
E[\Sigma_{n+1} \mid \mathcal{F}_n] &= \Sigma_n + E[X_{n+1} \mu_{n+1}\mid \mathcal{F}_n] 
= \left(1+\dfrac{p(\beta+1)}{n}\right) \Sigma_n.
\end{align*}
For $\xi > -1$, let
\begin{align}
 c_n(\xi) := \dfrac{\Gamma(n+\xi)}{\Gamma(n)\Gamma(\xi+1)} \sim \dfrac{n^{\xi}}{\Gamma(\xi+1)}\quad\mbox{as $n \to \infty$.}
 \label{def:c_n(xi)}
\end{align}
Note that $\mu_n = c_n(\beta)$. Put
\begin{align}
\label{def:MnMart}
M_n := \dfrac{\Sigma_n}{c_n(p(\beta+1))}.
\end{align}
Since $\{M_n\}$ is a non-negative martingale, there exists a non-negative random variable $M_{\infty}$ such that
\[ \lim_{n \to \infty}M_n 
= M_{\infty} \quad \mbox{a.s.} \]
As a consequence,
\begin{align}
E[\Sigma_n] = c_n(p(\beta+1)) \cdot E[\Sigma_1] = c_n(p(\beta+1)),
\label{eq:EofSigma_n} 
\end{align}
and
\begin{align}
E[\Sigma_n] \sim \dfrac{n^{p(\beta+1)}}{\Gamma(p(\beta+1))} \quad \mbox{as $n \to \infty$.} \label{eq:EofSigma_nAsymp} 
\end{align}

Our main results are the following.

\begin{theorem} \label{thm:EofXi_n} For $\beta>-1$, 
\begin{align*}
E[\Xi_n] = \begin{cases}
\dfrac{p(\beta+1)}{p(\beta+1)-\beta} \cdot \dfrac{c_n(p(\beta+1))}{c_n(\beta)} + \dfrac{\beta}{\beta-p(\beta+1)} &\mbox{if $\beta \neq \dfrac{p}{1-p}$} \\
\displaystyle  \sum_{k=0}^{n-1} \dfrac{\beta}{k+\beta}&\mbox{if $\beta = \dfrac{p}{1-p}$.}
\end{cases}
\end{align*}
\end{theorem}
\noindent 
As a corollary, from \eqref{def:c_n(xi)} we have
\begin{corollary} \label{cor:EofXi_n} (i) If $-1<\beta<p/(1-p)$ then
\begin{align*}
E[\Xi_n] \sim C(p,\beta) \cdot n^{p(\beta+1)-\beta} \quad \mbox{as $n \to \infty$,}
\end{align*}
where
\begin{align}
C(p,\beta):=\dfrac{1}{p(\beta+1)-\beta} \cdot \dfrac{\Gamma(\beta+1)}{\Gamma(p(\beta+1))}. \label{eq:DefCbeta}
\end{align}
\noindent (ii) If $\beta=p/(1-p)$ then
\begin{align*}
E[\Xi_n] \sim \beta \log n \quad \mbox{as $n \to \infty$.}
\end{align*}
\noindent (iii) If $\beta> p/(1-p)$ then 
\begin{align*}
E[\Xi_{\infty}] = \lim_{n \to \infty} E[\Xi_n] = \dfrac{\beta}{\beta-p(\beta+1)}.
\end{align*}
In particular, $\Xi_{\infty}<+\infty$ a.s.
\end{corollary}


\begin{theorem}
\label{r-thmsmallbeta}
If $-1<\beta<p/(1-p)$ then $P(M_{\infty}>0) > 0$, and 
$$
\Xi_n \sim  C(p,\beta) M_{\infty} n^{p (\beta+1)-\beta}\quad \mbox{on $\{M_{\infty}>0\}$.}
$$
\end{theorem}

\begin{theorem}
\label{r-thm:PositiveBetabdd} If $\beta >0$ then $P(\Xi_{\infty}=k)>0$ for any $k \in \mathbb{N}$, and in particular, $P(\mbox{$\Xi_n=1$ for all $n\in \mathbb{N}$}) >0$.
\end{theorem}

\noindent From Corollary \ref{cor:EofXi_n} (ii) we have $E[\Xi_{\infty}]=+\infty$ for $\beta = p/(1-p)$, however

\begin{theorem}
\label{r-thm:Extinction}  If $\beta = p/(1-p)$ then $P(\Xi_{\infty} <+\infty)=1$.
\end{theorem}

\noindent For the case when $\beta<0$, later in Lemma \ref{lem:NegativeBetaUnbdd} we obtain a lower bound on the rate of growth of $\Xi_n$, and show $P(\Xi_{\infty}=+\infty)=1$.

\section{Proof of Theorem \ref{thm:EofXi_n}}
\label{r-sec:proofs}

First we show that (\ref{r-beta1}) is indeed a probability mass function.
Assume that $a>-1$, $b \geq 0$ and $b\neq a+1$.
Note that
\begin{align*}
 \dfrac{\Gamma(k+a)}{\Gamma(k+b)} &= \dfrac{\Gamma(k+a)}{\Gamma(k+b)} \cdot \dfrac{(k+a)-(k+b-1)}{a-b+1} \\
 &=  \dfrac{1}{a-b+1} \left( \dfrac{\Gamma(k+1+a)}{\Gamma(k+b)} - \dfrac{\Gamma(k+a)}{\Gamma(k-1+b)}\right)
\end{align*}
for $k \in \mathbb{N}$. For $N,N' \in \mathbb{N}$ with $N<N'$, we have
\begin{align}
\sum_{k=N}^{N'} \dfrac{\Gamma(k+a)}{\Gamma(k+b)} &= \dfrac{1}{a-b+1}\left( \dfrac{\Gamma(N'+1+a)}{\Gamma(N'+b)} - \dfrac{\Gamma(N+a)}{\Gamma(N-1+b)}\right), \label{lem:Bercu18LemmaSeed}
\end{align}
where we regard $1/\Gamma(0)=0$.
By \eqref{lem:Bercu18LemmaSeed}, 
\begin{align*}
 \sum_{k=1}^n \mu_k &= \dfrac{1}{\Gamma(\beta+1)}\sum_{k=1}^n \dfrac{\Gamma(k+\beta)}{\Gamma(k)} \\
 &= \dfrac{1}{\Gamma(\beta+1)} \cdot \dfrac{1}{\beta+1} \left( \dfrac{\Gamma(n+1+\beta)}{\Gamma(n)} - \dfrac{\Gamma(1+\beta)}{\Gamma(0)} \right) 
 =\mu_{n+1} \cdot \dfrac{n}{\beta+1},
\end{align*}
which is equivalent to
\[
\sum_{k=1}^n P(\beta_{n+1}=k) = \sum_{k=1}^n \dfrac{\beta + 1}{n} \cdot \dfrac{\mu_k}{\mu_{n+1}} =1.
\]

To prove Theorem \ref{thm:EofXi_n} we need the following
\begin{lemma} \label{lem:Bercu18LemmaVariant} Let $x,y >-1$. If $x \neq y$ then
\begin{align*}
\sum_{k=1}^{n-1} \dfrac{c_k(x)}{k \cdot c_{k+1} (y)} &= \dfrac{1}{x-y} \cdot \left\{ \dfrac{c_n(x)}{c_n(y)} - 1 \right\}.
\end{align*}
If $x=y$ then
\begin{align*}
\sum_{k=1}^{n-1} \dfrac{c_k(x)}{k \cdot c_{k+1} (y)} &= \sum_{k=1}^{n-1} \dfrac{1}{k+x}.
\end{align*}
\end{lemma}

\begin{proof} By \eqref{def:c_n(xi)},
\begin{align*}
\sum_{k=1}^{n-1} \dfrac{c_k(x)}{k \cdot c_{k+1} (y)} &= \sum_{k=1}^{n-1} \dfrac{1}{k} \cdot \dfrac{\Gamma(k+x)}{\Gamma(k)\Gamma(x+1)} \cdot \dfrac{\Gamma(k)\Gamma(y+1)}{\Gamma(k+1+y)} \\
&= \dfrac{\Gamma(y+1)}{\Gamma(x+1)} \sum_{k=1}^{n-1} \dfrac{\Gamma(k+x)}{\Gamma(k+y+1)}.
\end{align*}
We obtain the conclusion if $x=y$. Otherwise by \eqref{lem:Bercu18LemmaSeed}
\begin{align*}
\sum_{k=1}^{n-1} \dfrac{c_k(x)}{k \cdot c_{k+1} (y)} &= \dfrac{\Gamma(y+1)}{\Gamma(x+1)} \cdot \dfrac{1}{x-y} \cdot \left\{ \dfrac{\Gamma(n+x)}{\Gamma(n+y)} - \dfrac{\Gamma(1+x)}{\Gamma(1+y)}\right\} \\
&=  \dfrac{1}{x-y} \cdot\left\{ \dfrac{c_n(x)}{c_n(y)} - 1 \right\}.
\end{align*} 
This completes the proof.
\end{proof}

\begin{proof}[Proof of Theorem \ref{thm:EofXi_n}]
Noting that by \eqref{X_n+1} and \eqref{eq:EofSigma_n},
\begin{align*}
E[\Xi_{n+1}] - E[\Xi_n] = E[X_{n+1}] = \dfrac{p(\beta+1) \cdot E[\Sigma_n]}{n \cdot \mu_{n+1}} = \dfrac{p(\beta+1) \cdot c_n(p(\beta+1))}{n \cdot c_{n+1}(\beta)},
\end{align*}
we have
\begin{align*}
E[\Xi_n] &= E[\Xi_1] + \sum_{k=1}^{n-1} \dfrac{p(\beta+1) \cdot c_k(p(\beta+1))}{k \cdot c_{k+1}(\beta)}.
\end{align*}
This together with Lemma \ref{lem:Bercu18LemmaVariant} implies Theorem \ref{thm:EofXi_n}. 
\end{proof}

\section{Proof of Theorem \ref{r-thmsmallbeta}}
\label{r-sec:proof-2}

Regarding the non-negative martingale $M_n=\Sigma_n/c_n(p(\beta+1))$, obtained in Section \ref{sec:results}, we have

\begin{proposition} \label{thm:MartL2bdd} $\{M_n : n \in \mathbb{N} \}$ is an $L^2$-bounded martingale if and only if $\beta < p/(1-p)$.
In particular, if $\beta < p/(1-p)$ then $P(M_{\infty}>0)>0$.
\end{proposition}

\begin{proof} Using \eqref{X_n+1},
\begin{align*}
E[\Sigma_{n+1}^2 \mid \mathcal{F}_n] 
=\left(1+\dfrac{2p(\beta+1)}{n} \right) \cdot \Sigma_n^2 + \mu_{n+1} \cdot \dfrac{p(\beta+1)}{n} \cdot \Sigma_n.
\end{align*}
Setting
\[ L_n := \dfrac{\Sigma_n^2}{c_n(2p(\beta+1))}, \]
we have, for some $C(\beta)>0$,
\begin{align}
E[L_{n+1}] - E[L_n] &= c_{n+1}(\beta) \cdot \dfrac{p(\beta+1) \cdot E[\Sigma_n]}{n \cdot c_{n+1}(2p(\beta+1))} \notag \\
&\sim C(\beta) \cdot n^{\beta-p(\beta+1)-1} \quad\mbox{as $n \to \infty$.} \label{eq:MartL2bddKey}
\end{align}
Note that
\begin{align*}
\dfrac{E[M_n^2]}{E[L_n]} = \dfrac{c_n(2p(\beta+1))}{c_n(p(\beta+1))^2} \sim \dfrac{\Gamma(p(\beta+1)+1)^2}{\Gamma(2p(\beta+1)+1)} \quad \mbox{as $n \to \infty$.} 
\end{align*}
Using \eqref{eq:MartL2bddKey}, we see that $\displaystyle \sup_{n \geq 1} E[M_n^2]<+\infty$ if and only if $\beta-p(\beta+1)<0$.
\end{proof}

Theorem \ref{r-thmsmallbeta} follows from Proposition \ref{thm:MartL2bdd} and the following lemma.

\begin{lemma} \label{lem:ConsequencesCondBC} If $-1<\beta<p/(1-p)$, then 
$$
\Xi_n \sim  \dfrac{\Gamma(\beta+1) M_{\infty}}{\Gamma(p(\beta+1))} \cdot \dfrac{n^{p (\beta+1)-\beta}}{p (\beta+1)-\beta} \quad \mbox{on $\{M_{\infty}>0\}$.}
$$
\end{lemma}

\begin{proof}
Recall (\ref{X_n+1}), that
\[ E[X_{n+1} \mid \mathcal{F}_n] = \dfrac{p(\beta+1)}{n \mu_{n+1}} \cdot \Sigma_n \quad \mbox{and}\quad \mu_{n+1} 
  \sim \dfrac{n^{\beta}}{\Gamma(\beta+1)}.  \]
On $\{M_{\infty}>0\}$, 
\begin{align*}
E[X_{n+1} \mid \mathcal{F}_n] &\sim \dfrac{p(\beta+1)}{n} \cdot \dfrac{\Gamma(\beta+1)}{n^{\beta}} \cdot \dfrac{M_{\infty} \cdot n^{p (\beta+1)}}{\Gamma(p(\beta+1)+1)} \\
&= \dfrac{\Gamma(\beta+1) M_{\infty}}{\Gamma(p(\beta+1))} \cdot n^{p (\beta+1)-\beta-1}\quad \mbox{as $n \to \infty$.}
\end{align*}
Define
\begin{align}
\label{eq:AnDefn}
A_n:= \sum_{k=1}^n E[X_k \mid \mathcal{F}_{k-1}].
\end{align}
From the conditional Borel--Cantelli lemma (see e.g. Williams \cite{Williams91Book}, p.124), 
\begin{itemize}
\item $\Xi_{\infty} < + \infty$ a.s. on $\{A_{\infty}<+\infty\}$.
\item $\Xi_n \sim A_n$ as $n \to \infty$, a.s. on $\{A_{\infty}=+\infty\}$.
\end{itemize}
If $-1 < \beta < p/(1-p)$ then 
\begin{align*}
 A_n \sim  \dfrac{\Gamma(\beta+1) M_{\infty}}{\Gamma(p(\beta+1))} \cdot \dfrac{n^{p (\beta+1)-\beta}}{p (\beta+1)-\beta} \quad \mbox{on $\{M_{\infty}>0\}$,} 
\end{align*}
which completes the proof.
\end{proof}

\begin{remark}
\label{rem:ConsequencesCondBC} 
If $\beta>p/(1-p)$, i.e. $p (\beta+1)-\beta-1<-1$, then $A_{\infty} < +\infty$ and $\Xi_{\infty} <+\infty$ a.s. (This is another proof of Corollary \ref{cor:EofXi_n} (iii).) When $\beta=p/(1-p)$ we have 
\[
\Xi_n \sim  \beta M_{\infty} \log n  \quad \mbox{on $\{M_{\infty}>0\}$,}
\]
but as we will see later $P(M_{\infty}>0)=0$ in this case.
\end{remark}

 
As mentioned at the end of Section \ref{sec:results}, we close this section with

\begin{lemma} \label{lem:NegativeBetaUnbdd} If $-1<\beta <0$, then there exists a positive constant $C=C(\beta,p)$ such that
\[
P\left(
\mbox{$A_n \geq C n^{-\beta}$ for all $n$, and $\Xi_n \sim A_n$ as $n \to \infty$}\\
\right) = 1.
\]
\end{lemma} 
\begin{proof} Recall that $P(X_1=1)=1$ and $E[X_1]=1$. 
For each $k=2,3,\ldots$,
\begin{align*}
E[X_k \mid \mathcal{F}_{k-1}] \geq p \cdot P(\beta_k=1) = \dfrac{p(\beta+1)}{k\mu_{k+1}}. 
\end{align*}
The conclusion follows from $\mu_{k+1} \sim k^{\beta}/\Gamma(\beta+1)$ as $k \to \infty$.
\end{proof}


\section{Proof of Theorem \ref{r-thm:PositiveBetabdd}}


Throughout this section we assume that $\beta>0$. First we treat the special case $k=1$. Note that
\begin{align*}
P(X_2=0) &= 1-p \cdot P(\beta_2 =1) = 1-p,
\end{align*}
and
\begin{align*}
P(X_n=0 \mid X_2=\cdots=X_{n-1}=0) = 1- p \cdot P(\beta_n =1)
\end{align*}
for $n=2,3,\ldots$. From \eqref{r-beta1},
\begin{align*}
P(\beta_n =1) = \dfrac{\beta+1}{n-1} \cdot \dfrac{1}{\mu_n} \sim \dfrac{\beta+1}{\Gamma(\beta+1)} \cdot n^{-\beta-1} \quad \mbox{as $n \to \infty$,}
\end{align*}
and we have 
\begin{align*}
P(\Xi_{\infty}=1)=P(\mbox{$\Xi_n=1$ for all $n \in \mathbb{N}$}) 
\geq \prod_{n=2}^{\infty} \{1- p \cdot P(\beta_n =1)\} > 0. 
\end{align*}
For the general case, fix $k \in \mathbb{N}$. First note that
\[ P(X_1=\cdots=X_k=1) \geq \prod_{j=2}^k p \cdot P(\beta_j=1) >0. \]
On the other hand,
\begin{align*}
 P(X_{k+1}=0 \mid X_1=\cdots=X_k=1) &= 1 - p \cdot \sum_{j=1}^k P(\beta_{k+1}=j) \\
 &= 1- \dfrac{p(\beta+1)}{k} \cdot \dfrac{\sum_{j=1}^k \mu_j}{\mu_{k+1}},
\end{align*}
and
\begin{align*}
 &P(X_n=0 \mid X_1=\cdots=X_k=1, \, X_{k+1}=\cdots=X_{n-1}=0) \\
 &\geq 1 - p \cdot \sum_{j=1}^k P(\beta_n=j) = 1- \dfrac{p(\beta+1)}{n-1} \cdot \dfrac{\sum_{j=1}^k \mu_j}{\mu_n}
\end{align*}
for $n \geq k+2$. Thus we have $P(\Xi_{\infty} = k)>0$.
\qed


\section{Proof of Theorem \ref{r-thm:Extinction}}
\label{r-sec:coupling}

Throughout this section we assume that $\beta>0$.
We introduce a multi-type branching process with rate $\{ q(x,y) : x,y \in \mathbb{N}^2 \cap \{x<y\}\}$, where
\begin{align}
q(x,y)= p \cdot P(\beta_y =x) = \frac{p(\beta+1)}{y-1} \cdot\frac{\mu_x}{\mu_y},
\end{align}
where $\mu_n$ was defined in \eqref{r-mu1}.
The process 
is defined as follows:

\vskip 3mm
\begin{itemize}
\item[(1)] The first generation consists of one particle of type $y^{(1)} \equiv 1$. 
\item[(2)] The particle of type $y^{(1)}$ in the first generation gives birth to a second generation particle of type $y^{(2)}> y^{(1)}$ with probability $q(y^{(1)},y^{(2)})$. 
The events of the existence of different types of children and their numbers are independent within themselves as well as independent of each other.
\item[(3)] In case there is no particle in the second generation, we stop. Otherwise each particle of the second generation with type $y^{(2)}$ gives birth to a third generation particle of type $y^{(3)}>y^{(2)}$ with probability $q(y^{(2)}, y^{(3)})$.
The events of the existence of different types of children and their numbers are independent within themselves as well as independent of each other.
\item[(4)] In case there is no particle in the $k$-th generation, we stop.
Otherwise each particle of type $y^{(k)}$ in the $k$-th generation gives birth to a $(k+1)$-th generation particle of type $y^{(k+1)} > y^{(k)}$ with probability $q(y^{(k)}, y^{(k+1)})$,
where $y^{(k)}$ is the type of the parent from the  $k$-th generation. 
Independence of the number and type of children hold as earlier. 
\end{itemize}

To exhibit the coupling of $\{\Xi_n : n \in \mathbb{N}\}$ and the branching process we introduce a modified model. Let $\{ \widehat{\beta}(i,j) : 1\le i <j <\infty\}$ be a collection of $\{0,1\}$-valued independent random variables such that
$$
P\left(\widehat{\beta}(i,j)=1\right)= P(\beta_{j}= i).
$$
We also assume that $\widehat{\beta}(i,j)$ is independent of all random processes considered so far. Let $Y_1=1$ and 
\begin{align}
Y_{n+1} =\begin{cases} 
\displaystyle{\max_{1\le i \le n}}\widehat{\beta}(i,n+1)Y_i  & \text{with probability $p$}
\\
0 & \text{with probability $1-p$.}
\end{cases}
\end{align}
We put 
\[ \widehat{\xi}_n := \{ k \in \mathbb{N} : k \le n,\,Y_k=1  \} \]
and 
\[ \widehat{\Xi}_n := \# \widehat{\xi}_n = \sum_{k=1}^n Y_k. \]
Using the Skorokhod representation theorem and coupling the processes $\widehat{\xi}_n$ and the first $n$ generations of the branching process, we have
\begin{align}\label{;Eb}
\widehat{\Xi}_n  \le \# \left\{ k \in \mathbb{N} :
\begin{array}{@{\,}c@{\,}}
\text{there is a particle of type $k$} \\
\text{in the first $n$ generations} \\
\end{array} \right\}\quad \text{a.s.}
\end{align}
Thus the extinction of the branching process implies 
\[ \widehat{\Xi}_\infty := \lim_{n\to\infty}\widehat{\Xi}_n < \infty\quad \text{a.s.} \]

In this context we first observe that for $\beta>0$, the expected number of children of a particle irrespective of its  type  is a constant. 

\begin{lemma} \label{lem:BranchingMean} Assume that $\beta>0$. For each $k \in \mathbb{N}$, 
\[ \sum_{x=k+1}^{\infty} q(k,x) =  \dfrac{p(\beta+1)}{\beta}. \]
\end{lemma}

\begin{proof}
For any $K>k$, we have
\begin{align*}
\sum_{x=k+1}^K q(k,x) &= p(\beta+1) \cdot \sum_{x=k+1}^K \dfrac{1}{x-1} \cdot \dfrac{\Gamma(x)\Gamma(\beta+1)}{\Gamma(x+\beta)} \cdot \dfrac{\Gamma(k+\beta)}{\Gamma(k)\Gamma(\beta+1)} \\
&=p(\beta+1) \cdot \dfrac{\Gamma(k+\beta)}{\Gamma(k)} \sum_{x=k+1}^K \dfrac{\Gamma(x-1)}{\Gamma(x+\beta)} \\
&=p(\beta+1) \cdot \dfrac{\Gamma(k+\beta)}{\Gamma(k)} \cdot  \dfrac{1}{-\beta} \left( \dfrac{\Gamma(K)}{\Gamma(K+\beta)} - \dfrac{\Gamma(k)}{\Gamma(k+\beta)} \right) \\
&=\dfrac{p(\beta+1)}{\beta} \cdot \left(1-\dfrac{\Gamma(k+\beta)}{\Gamma(k)} \cdot \dfrac{\Gamma(K)}{\Gamma(K+\beta)}\right),
\end{align*}
where we used \eqref{lem:Bercu18LemmaSeed}. Noting that $\Gamma(K)/\Gamma(K+\beta) \sim K^{-\beta}$ as $K \to \infty$,
we obtain the desired conclusion.
\end{proof}


\begin{remark}
Although the expected number $p(\beta+1)/\beta$ of children of a particle is smaller or equal to one if and only if $\beta \geq p/(1-p)$, we still need to prove that the branching process dies out in this case, because the progeny distribution is different and depends on the type.
\end{remark}

Let
\[ p_0^{(k)} := \prod_{x=k+1}^{\infty} \{1-q(k,x)\} \quad \mbox{for $k\in \mathbb{N}$.}  \]

\begin{lemma} \label{lem:BranchingHideki230804_A}
If $\beta>0$ then $\displaystyle \inf_{k\in \mathbb{N}}  p_0^{(k)} > 0$.
\end{lemma}

\begin{proof} Note that 
\[ 0 \leq q(k,x) \leq p \quad \mbox{for any $k\in \mathbb{N}$ and $x=k+1,k+2,\ldots$.} \]
Let 
\[ C=C(p) := \dfrac{-\log(1-p)}{p}. \]
It is straightforward to see that
\[ 1-t \geq e^{-Ct} \quad \mbox{for $t \in [0,p]$,} \] 
which implies
\[
p_0^{(k)} \geq \exp\left\{-C\sum_{x=k+1}^{\infty} q(k,x)\right\} =  \exp\left\{-\dfrac{Cp(\beta+1)}{\beta}\right\}
\]
for all $k\in \mathbb{N}$.
\end{proof}


Let $N_n$ be the number of particles in the $n$-th generation of the branching process. 
Assume that $\beta>0$, and so $m:= p(\beta+1)/\beta< \infty$. We have
\[ E[N_n] = m^{n-1} \quad \mbox{for $n\in \mathbb{N}$.} \]

\begin{lemma} \label{lem:BranchingHideki230804_B} If $\beta \geq p/(1-p)$ then $\displaystyle \lim_{n \to \infty} P(N_n>0)=0$.
\end{lemma}

\begin{proof}
If $\beta>p/(1-p)$ then $m <1$ by Lemma \ref{lem:BranchingMean}, and by Markov's inequality, 
\[ P(N_n>0) \leq E[N_n] = m^{n-1} \to 0 \quad \mbox{as $n \to \infty$.} \]
Hereafter we assume that $\beta=p/(1-p)$, i.e. $m=1$. For any $K \in \mathbb{N}$, again by Markov's inequality,
\begin{align*}
 P(N_n>0) &= P(0<N_n \leq K) + P(N_n>K) \\
 &\leq P(0<N_n \leq K) + \dfrac{1}{K+1}.
\end{align*}
To obtain the desired conclusion it suffices to show that
\[ \mbox{for any $K \in \mathbb{N}$}, \quad  \lim_{n \to \infty} P(0<N_n \leq K)=0. \]
Assume that
\[ \limsup_{n \to \infty} P(0<N_n \leq K) = \delta > 0.\]
We can find a subsequence $\{n_j\}$ with $1 \leq n_1 < n_2<\cdots$ and 
\[P(0<N_{n_j} \leq K) \geq \dfrac{\delta}{2} \quad \mbox{for $j \in \mathbb{N}$.}  \]
For $j \in \mathbb{N}$, define
\[ A_j := \{ 0<N_{n_j} \leq K, \, N_{n_j+1} = 0 \}.\]
Noting that $\{A_j : j \in \mathbb{N}\}$ are mutually disjoint, 
\begin{align*}
1 \geq P\left( \bigcup_{j=1}^{\infty} A_j \right) = \sum_{j=1}^{\infty} P(A_j) 
\geq \sum_{j=1}^{\infty} \dfrac{\delta}{2} \cdot \left( \inf_{k \in \mathbb{N}}  p_0^{(k)}\right)^K = +\infty.
\end{align*}
This is a contradiction.
\end{proof}

Thus, from Lemma \ref{lem:BranchingHideki230804_B}, we  have that the multi-type branching process dies out if and only if $m \le 1$, i.e. $\beta \ge p/(1-p)$.
Hence,
\begin{align}
P\left(\widehat{\Xi}_\infty  < \infty\right) =1 \quad \text{whenever $\beta \ge \frac{p}{1-p}$. }
\end{align}
For any $n\geq 1$, we have that $\# \widehat{\xi}_n \leq n$ a.s. and so 
\begin{align}
P\left( \left. \widehat{\Xi}_\infty  < \infty \,\right|\,\widehat{\xi}_n\right)=1\quad \text{a.s.}
\end{align}

Next, we consider the relation between $\{\Xi_n\}$ and $\{\widehat{\Xi}_n\}$.
We put 
\[ B_i = B_{i,n}:= \{ \widehat{\beta}(i,n+1)=1\}. \]
Then using the inclusion-exclusion formula, for any $I \subset \{1,\ldots,n\}$,
\[
0 \le \sum_{i\in I}P(B_i)- P\left(\bigcup_{i\in I}B_i\right) \le \sum_{i,j\in I: i\not=j}P(B_i\cap B_j), 
\]
and noting that 
$$
\left\{\max_{i\in I}  \widehat{\beta}(i,n+1) =1\right\} = \bigcup_{i\in I}B_i, \quad \{\beta_{n+1}\in I\} = \bigcup_{i \in I} \{ \beta_{n+1} =i \},
$$
and $P(B_i)=P(\beta_{n+1} =i)$,
we have
\[
0 \le  P( \beta_{n+1}\in I) - P\left(\displaystyle{\max_{i\in I}  \widehat{\beta}(i,n+1) =1} \right) \le \sum_{i,j\in I: i\not=j}P(B_i\cap B_j). 
\]
Now
\begin{align*}
\sum_{i,j\in I: i\not=j}P(B_i\cap B_j) &= \sum_{i,j \in I: i\not=j}P(B_i)P(B_j)\\
& \leq \left(\# I\right)^2\left(\frac{p(\beta+1)}{n\mu_{n+1}}\max_{1\le i \le n}\mu_i\right)^2\\
& \leq \frac{\left(\# I \right)^2}{n^2} \left(p(\beta+1)\right)^2 \\
& \leq \frac{\left(p(\beta+1)\right)^2 }{n^{1+\delta}}\quad \text{if } \# I \leq n^{(1-\delta)/2} \text{ for some } \delta > 0.
\end{align*}
Thus, for $I \subset \mathbb{N}$ with $\# I \le n^{(1-\delta)/2}$ for some $\delta >0$, we have
\[
0 \leq P( \beta_{n+1}\in I) - P\left(\max_{i\in I} \widehat{\beta}(i,n+1)=1 \right) \le \frac{\left(p(\beta+1)\right)^2}{n^{1+\delta}}.
\]
If 
\begin{align}\label{;2.7cond}
\sum_{n=1}^\infty P(\Xi_n \ge n^{(1-\delta)/2}) 
< \infty,
\end{align}
then for any $\varepsilon \in (0, 1)$,  there exists $N \in \mathbb{N}$ such that
\begin{align}\label{;2.7}
P( \Xi_n \le n^{(1-\delta)/2}  \text{ for all $n \ge N$} ) \ge 1-\varepsilon.
\end{align}
Hence, on $\Omega_{N}:=\{   \Xi_n \le n^{(1-\delta)/2}  \text{ for all $n \ge N$}  \}$,
\begin{align}
&\sum_{n=N}^\infty \left\{ P( \beta_{n+1}\in \xi_n \mid \xi_n) - P\left( \left. \max_{i\in \xi_n} \widehat{\beta}(i,n+1)=1 \,\right|\, \xi_n\right) \right\} \notag \\
&\le \sum_{n=N}^\infty \frac{\left(p(\beta+1)\right)^2}{n^{1+\delta}}, \label{;2s}
\end{align}
where $\xi_n = \{ k \in \mathbb{N}:k \le n ,\,  X_k=1 \}$.
Let $\widehat{\xi}_n$ be as above.
We now obtain a relation between
$\{ (\xi_n, X_n) \}_{n\in\mathbb{N}}$ and $ \{(\widehat{\xi}_n, Y_n)\}_{n\in\mathbb{N}}$.  Note that,
for any $n\in \mathbb{N}$ and $\zeta \subset \{1,2,\dots,n\}$,
\begin{align}\notag
&P( X_{n+1}\not= Y_{n+1} \mid \xi_n = \widehat{\xi}_n=\zeta) 
\\ \label{;c_X}
&= P( \beta_{n+1}\in\zeta) 
- P\left( \max_{i \in \zeta} \widehat{\beta}(i,n+1)=1 \right),
\end{align}
and
\begin{align*}
&P(\xi_{n+i} = \widehat{\xi}_{n+i} \mbox{ for all $i\ge 1$} \mid \xi_n = \widehat{\xi}_n=\zeta)
\\
&=P(\xi_{n+i} = \widehat{\xi}_{n+i}\mbox{ for all $i\ge 2$}\mid \xi_{n+1}=\widehat{\xi}_{n+1}) P(X_{n+1}= Y_{n+1}\mid \xi_n = \widehat{\xi}_n=\zeta)
\\
&=P(X_{n+1}= Y_{n+1}\mid \xi_n = \widehat{\xi}_n=\zeta)\prod_{i=1}^\infty P(X_{n+i+1}= Y_{n+i+1}\mid \xi_{n+i} = \widehat{\xi}_{n+i}).
\end{align*}
Thus, from (\ref{;2s}) and (\ref{;c_X}), for any $\varepsilon >0$,
we can find $N_1\in \mathbb{N}$ such that, for any $n\ge N_1$,
\begin{align}
P(\xi_{n+i} = \widehat{\xi}_{n+i} \mbox{ for all $i \ge 1$} \mid \xi_n = \widehat{\xi}_n =\zeta) \ge 1-\varepsilon.
\end{align}
Hence
\begin{align*}
&P(\Xi_{\infty} <\infty\mid  \xi_{N_1}=\zeta)
\\
&\ge P(\Xi_{\infty} <\infty, \xi_\ell = \widehat{\xi}_\ell \mbox{ for all $\ell > N_1$} \mid  \xi_{N_1}=\widehat{\xi}_{N_1} =\zeta)
\\
&= P(\Xi_{\infty} <\infty, \xi_\ell = \widehat{\xi}_\ell \mbox{ for all $\ell > N_1$} \mid  \xi_{N_1}=\widehat{\xi}_{N_1} =\zeta)
\\
&\ge P(\Xi_{\infty} <\infty \mid \widehat{\xi}_{N_1} =\zeta)
- P(\xi_\ell \not= \widehat{\xi}_\ell \text{ for some $\ell > N_1$} \mid \xi_{N_1}=\widehat{\xi}_{N_1} =\zeta)
\\
&\ge 1- \varepsilon.
\end{align*}
From (\ref{;2.7}) and the above we have
\begin{align*}
P(\Xi_{\infty} < \infty)
&\ge \sum_{\zeta: \# \zeta \le N_1^{(1-\delta)/2}} 
P(\Xi_{\infty} <\infty\mid  \xi_{N_1}=\zeta) P(\xi_{N_1}=\zeta)
\ge (1-\varepsilon)^2.
\end{align*}
As $\varepsilon>0$ is arbitrary, we complete the proof of $P(\Xi_{\infty} < \infty)=1$ under \eqref{;2.7cond}.
We will later see in \eqref{eq:asympE[(Xi_n)^3]crit} of Section \ref{sec:MomentsCritical}, for the critical case $\beta=p/(1-p)$, \eqref{;2.7cond} is satisfied. \qed

\section{Moments for the critical case}
\label{sec:MomentsCritical}


\begin{proposition} \label{prop:MixedMomentsCritical} Assume that $p \in (0,1)$ and $\beta=p/(1-p)$. For $k \in \{1,2,3\}$ and $\ell \in \{0,1,\ldots,k\}$, we have that
\begin{align} \label{eq:MixedMomentCritical}
E\left[ \left(\Xi_n\right)^{k-\ell} \left(\Sigma_n\right)^{\ell} \right] \sim C_{k,\ell} \cdot n^{\ell\beta} (\log n)^{2k-1-\ell} \quad \mbox{as $n \to \infty$,}
\end{align}
where $C_{k,\ell}$ is a positive constant.
\end{proposition}


To prove Proposition \ref{prop:MixedMomentsCritical}, we need the following lemma.

\begin{lemma} \label{lem:ForMixedMomentCritical} Let $\xi \geq 0$. Assume that a sequence $\{x_n\}$ satisfies 
\begin{align}
x_1=c,\quad x_{n+1} = \left(1+\dfrac{\xi}{n}\right) x_n + f_n\quad \mbox{for $n \in \mathbb{N}$,} \label{Recursion:ForMixedMomentCritical}
\end{align}
where 
\begin{align}
\dfrac{f_n}{c_{n+1}(\xi)} \sim \dfrac{C (\log n)^m}{n} \quad \mbox{as $n \to \infty$} \label{eq:ForMixedMomentCritical}
\end{align}
for some $C>0$ and $m \in \mathbb{Z}_+$. Then there exists a constant $K=K(\xi,m)>0$ such that
\begin{align*}
x_n \sim Kn^{\xi} (\log n)^{m+1}\quad \mbox{as $n \to \infty$.} 
\end{align*}
\end{lemma}

\begin{proof} Putting $y_n := x_n/c_{n+1}(\xi)$, we have that
\begin{align*}
y_n &= y_1 + \sum_{k=1}^{n-1} (y_{k+1}-y_k) 
= \dfrac{c}{c_{n+1}(\xi)} + \sum_{k=1}^{n-1} \dfrac{f_k}{c_{k+1}(\xi)}.
\end{align*}
From \eqref{eq:ForMixedMomentCritical},
\begin{align*}
y_n \sim \dfrac{C(\log n)^{m+1}}{m+1}\quad \mbox{as $n \to \infty$.}
\end{align*}
Since $c_{n+1}(\xi) \sim n^{\xi}/\Gamma(\xi+1)$ as $n \to \infty$, we obtain the desired conclusion. \end{proof}

\begin{proof}[Proof of Proposition \ref{eq:MixedMomentCritical}] Assume that $\beta=p/(1-p)$, and so $p(\beta+1)=\beta$.
For the case $k=1$, Eq. \eqref{eq:EofSigma_nAsymp} and Corollary \ref{cor:EofXi_n} (ii) implies \eqref{eq:MixedMomentCritical} with
\begin{align*}
C_{1,0} = \beta \quad \mbox{and} \quad C_{1,1} = \dfrac{1}{\Gamma(\beta)}.
\end{align*}

We turn to the case $k=2$. Since
\begin{align*}
E[(\Sigma_{n+1})^2 \mid \mathcal{F}_n] 
=\left(1+\dfrac{2\beta}{n} \right) \cdot (\Sigma_n)^2 + \dfrac{\beta \mu_{n+1}}{n} \cdot \Sigma_n,
\end{align*}
we have that $x_n=E[(\Sigma_n)^2]$ satisfies \eqref{Recursion:ForMixedMomentCritical} with $c=1$, $\xi=2\beta$ and
\begin{align*}
f_n = \dfrac{\beta \mu_{n+1} E[\Sigma_n]}{n}.
\end{align*}
From \eqref{eq:EofSigma_nAsymp}, we can see that \eqref{eq:ForMixedMomentCritical} holds with $m=0$. Lemma \ref{lem:ForMixedMomentCritical} implies that 
\begin{align}
E[(\Sigma_n)^2] \sim C_{2,2} n^{2\beta} \log n\quad \mbox{as $n \to \infty$.}
\label{eq:asympE[(Sigma_n)^2]crit}
\end{align}
To obtain the asymptotics of $E[\Xi_n \Sigma_n]$, note that
\begin{align*}
E[\Xi_{n+1}\Sigma_{n+1} \mid \mathcal{F}_n] 
=\left(1+\dfrac{\beta}{n} \right) \cdot \Xi_n \Sigma_n + \dfrac{\beta}{n \mu_{n+1}} \cdot (\Sigma_n)^2 + \dfrac{\beta}{n} \cdot \Sigma_n.
\end{align*}
We see that $x_n=E[\Xi_n \Sigma_n]$ satisfies \eqref{Recursion:ForMixedMomentCritical} with $c=1$, $\xi=\beta$ and
\begin{align*}
f_n = \dfrac{\beta E[(\Sigma_n)^2]}{n \mu_{n+1}} + \dfrac{\beta E[\Sigma_n]}{n}.
\end{align*}
From \eqref{eq:EofSigma_nAsymp} and \eqref{eq:asympE[(Sigma_n)^2]crit}, we have that \eqref{eq:ForMixedMomentCritical} holds with $m=1$. Lemma \ref{lem:ForMixedMomentCritical} implies that 
\begin{align}
E[\Xi_n \Sigma_n] \sim C_{2,1} n^{\beta} (\log n)^2\quad \mbox{as $n \to \infty$.}
\label{eq:asympE[Xi_nSigma_n]crit}
\end{align}
As for $E[(\Xi_n)^2]$, we have
\begin{align*}
E[(\Xi_{n+1})^2 \mid \mathcal{F}_n] 
=(\Xi_n)^2 + \dfrac{2\beta}{n\mu_{n+1}} \cdot \Xi_n \Sigma_n + \dfrac{\beta}{n\mu_{n+1}} \cdot \Sigma_n.
\end{align*}
Since $x_n=E[(\Xi_n)^2]$ satisfies $c=1$, $\xi=0$ and
\begin{align*}
f_n = \dfrac{2\beta E[\Xi_n \Sigma_n]}{n\mu_{n+1}} + \dfrac{\beta E[\Sigma_n]}{n\mu_{n+1}}.
\end{align*}
From \eqref{eq:EofSigma_nAsymp} and \eqref{eq:asympE[Xi_nSigma_n]crit}, we can see that \eqref{eq:ForMixedMomentCritical} holds with $m=2$. Lemma \ref{lem:ForMixedMomentCritical} implies that 
\begin{align}
E[(\Xi_n)^2] \sim C_{2,0}  (\log n)^3\quad \mbox{as $n \to \infty$.}
\label{eq:asympE[(Xi_n)^2]crit}
\end{align}

The case $k=3$ can be handled in a similar manner. 
As 
\begin{align*}
E[(\Sigma_{n+1})^3 \mid \mathcal{F}_n] 
&=\left(1+\dfrac{3\beta}{n} \right) \cdot (\Sigma_n)^3 + \dfrac{3\beta\mu_{n+1}}{n} \cdot (\Sigma_n)^2+  \dfrac{3\beta\mu_{n+1}^2}{n} \cdot \Sigma_n,
\end{align*}
we have that $x_n=E[(\Sigma_n)^3]$ satisfies \eqref{Recursion:ForMixedMomentCritical} with $c=1$, $\xi=3\beta$ and
\begin{align*}
f_n = \dfrac{3\beta\mu_{n+1}E[(\Sigma_n)^2]}{n} +  \dfrac{3\beta\mu_{n+1}^2 E[\Sigma_n]}{n}.
\end{align*}
From \eqref{eq:EofSigma_nAsymp} and \eqref{eq:asympE[(Sigma_n)^2]crit}, we can see that \eqref{eq:ForMixedMomentCritical} holds with $m=1$. Lemma \ref{lem:ForMixedMomentCritical} implies that 
\begin{align}
E[(\Sigma_n)^3] \sim C_{3,3}  n^{3\beta} (\log n)^2 \quad \mbox{as $n \to \infty$.}
\label{eq:asympE[(Sigma_n)^3]crit}
\end{align}
Next, noting that
\begin{align*}
E[\Xi_{n+1}  (\Sigma_{n+1})^2 \mid \mathcal{F}_n] 
&=\left(1+\dfrac{2\beta}{n} \right) \cdot \Xi_n(\Sigma_n)^2 + \dfrac{\beta \mu_{n+1}}{n} \cdot \Xi_n \Sigma_n \\
&\quad + \dfrac{\beta}{n \mu_{n+1}} \cdot (\Sigma_n)^3 +\dfrac{2\beta}{n} \cdot (\Sigma_n)^2+   \cdot \dfrac{\beta \mu_{n+1}}{n} \cdot \Sigma_n,
\end{align*}
we have $x_n=E[\Xi_n(\Sigma_n)^2]$ satisfies \eqref{Recursion:ForMixedMomentCritical} with $c=1$, $\xi=2\beta$ and
\begin{align*}
f_n = \dfrac{\beta \mu_{n+1} E[\Xi_n \Sigma_n]}{n}
+ \dfrac{\beta E[(\Sigma_n)^3]}{n \mu_{n+1}}  +\dfrac{2\beta E[(\Sigma_n)^2]}{n} + \dfrac{\beta \mu_{n+1} E[\Sigma_n]}{n}.\end{align*}
From \eqref{eq:EofSigma_nAsymp}, \eqref{eq:asympE[(Sigma_n)^2]crit}, \eqref{eq:asympE[Xi_nSigma_n]crit} and \eqref{eq:asympE[(Sigma_n)^3]crit}, we can see that 
\eqref{eq:ForMixedMomentCritical} holds with $m=2$. Lemma \ref{lem:ForMixedMomentCritical} implies that 
\begin{align}
E[\Xi_n(\Sigma_n)^2] \sim C_{3,2}  n^{2\beta} (\log n)^3 \quad \mbox{as $n \to \infty$.}
\label{eq:asympE[Xi_n(Sigma_n)^2]crit}
\end{align}
Since
\begin{align*}
E[(\Xi_{n+1})^2  \Sigma_{n+1} \mid \mathcal{F}_n] 
&=\left(1+\dfrac{\beta}{n} \right) \cdot (\Xi_n)^2  \Sigma_n + \dfrac{2\beta}{n \mu_{n+1}} \cdot \Xi_n (\Sigma_n)^2 \\
&\quad  +   \dfrac{2\beta}{n} \cdot \Xi_n \Sigma_n
+    \dfrac{\beta}{n \mu_{n+1}}\cdot (\Sigma_n)^2 + \dfrac{\beta}{n}\cdot \Sigma_n,
\end{align*}
we have $x_n=E[(\Xi_n)^2 \Sigma_n]$ satisfies \eqref{Recursion:ForMixedMomentCritical} with $c=1$, $\xi=\beta$ and
\begin{align*}
f_n = \dfrac{2\beta E[\Xi_n (\Sigma_n)^2]}{n \mu_{n+1}} +   \dfrac{2\beta E[\Xi_n \Sigma_n]}{n} 
+    \dfrac{\beta E[(\Sigma_n)^2]}{n \mu_{n+1}}  + \dfrac{\beta E[\Sigma_n]}{n}.
\end{align*}
From \eqref{eq:EofSigma_nAsymp}, \eqref{eq:asympE[(Sigma_n)^2]crit}, \eqref{eq:asympE[Xi_nSigma_n]crit} and \eqref{eq:asympE[Xi_n(Sigma_n)^2]crit}, we can see that 
\eqref{eq:ForMixedMomentCritical} holds with $m=3$. Lemma \ref{lem:ForMixedMomentCritical} implies that 
\begin{align}
E[(\Xi_n)^2\Sigma_n] \sim C_{3,1}  n^{\beta} (\log n)^4 \quad \mbox{as $n \to \infty$.}
\label{eq:asympE[(Xi_n)^2Sigma_n]crit}
\end{align}
Finally, noting that
\begin{align*}
&E[(\Xi_{n+1})^3  \mid \mathcal{F}_n] \\
&=(\Xi_n)^3 +\dfrac{3\beta}{n \mu_{n+1}} \cdot (\Xi_n)^2 \Sigma_n  +   \dfrac{3\beta}{n \mu_{n+1}} \cdot \Xi_n \Sigma_n + \dfrac{\beta}{n \mu_{n+1}} \cdot \Sigma_n,
\end{align*}
we have $x_n=E[(\Xi_n)^3]$ satisfies \eqref{Recursion:ForMixedMomentCritical} with $c=1$, $\xi=0$ and
\begin{align*}
f_n = \dfrac{3\beta E[(\Xi_n)^2\Sigma_n]}{n \mu_{n+1}}   +  \dfrac{3\beta E[\Xi_n \Sigma_n ]}{n \mu_{n+1}} \cdot + \dfrac{\beta E[\Sigma_n]}{n \mu_{n+1}}.
\end{align*}
From \eqref{eq:EofSigma_nAsymp},
\eqref{eq:asympE[Xi_nSigma_n]crit} and \eqref{eq:asympE[(Xi_n)^2Sigma_n]crit}, we can see that 
\eqref{eq:ForMixedMomentCritical} holds with $m=4$. Lemma \ref{lem:ForMixedMomentCritical} implies that 
\begin{align}
E[(\Xi_n)^3] \sim C_{3,0} (\log n)^5 \quad \mbox{as $n \to \infty$.}
\label{eq:asympE[(Xi_n)^3]crit}
\end{align}
This completes the proof.
\end{proof}

\section{Conclusions}
\label{sec:refinements}

In Section \ref{r-sec:coupling} we introduced a comparison with the multi-type branching process.
Here we introduce another comparison with the LERW, and obtain some different bounds.

Assume that $-1<\beta<0$. By \eqref{eq:CondEX_n+1}, 
\begin{align}
P( X_{n+1} = 1 \mid \mathcal{F}_n) &= p \cdot  \dfrac{\beta+1}{n\mu_{n+1}} \cdot \sum_{k=1}^n X_k \mu_k \notag \\
&\geq \dfrac{p(\beta+1)}{n} \cdot \dfrac{\mu_n}{\mu_{n+1}}\cdot \sum_{k=1}^n X_k \notag \\
&= \dfrac{p(\beta+1)}{n} \cdot \dfrac{n}{n+\beta} \cdot \Xi_n \geq p(\beta+1) \cdot \dfrac{\Xi_n}{n}.
\label{eq:compareLERW}
\end{align}
Consider the LERW $\{\Xi'_n=\sum_{i=1}^n X'_i\}$ defined by
\begin{align}
\label{r-lazydefModified}
 X'_1 \equiv 1,\quad X'_{n+1} = \begin{cases}
X'_{U_n} & \text{with probability }p(\beta+1)\\
0 &  \text{with probability }1-p(\beta+1).
\end{cases}
\end{align}
We have
\begin{align}
  P( X'_{n+1}= 1 \mid \mathcal{F}'_n) = p(\beta+1) \cdot \dfrac{\Xi'_n}{n} \quad \mbox{for $n \in \mathbb{N}$},
  \label{eq:compareLERW_2}
\end{align}
where $\mathcal{F}'_n$ is the $\sigma$-algebra generated by $X'_1,\ldots,X'_n$. By \eqref{eq:compareLERW} and \eqref{eq:compareLERW_2}, we can construct a coupling such that with probability one,
\begin{align}
 \Xi_n \geq \Xi'_n
\quad \mbox{for all $n \in \mathbb{N}$,}
\label{eq:CouplingLERWneg}
\end{align}
and so with probability one, there exists a positive constant $C$ such that
\begin{align}
\label{r-Xi_n:bound}
\Xi_n \geq C n^{p(1+\beta)} \quad \mbox{for all $n \in \mathbb{N}$.}
\end{align}
Note that $-1<\beta<0$ and $p(1+\beta) > -\beta$ if and only if $-p/(1+p)<\beta<0$. The above argument gives a better lower bound than Lemma \ref{lem:NegativeBetaUnbdd}.

We believe that Lemma \ref{lem:NegativeBetaUnbdd} and the argument leading to \eqref{r-Xi_n:bound} can be combined to show that, for $-p/(1+p)<\beta<0$, with probability one, there exists a positive constant $C$ such that
\[ 
\Xi_n \geq C n^{p(1+\beta)-\beta} \quad \mbox{for all $n \in \mathbb{N}$.} \]

\begin{remark}
For $\beta>0$, using a similar argument leading to \eqref{eq:CouplingLERWneg} we can show that, with probability one,
\[ \Xi_n \leq \Xi'_n
\quad \mbox{for all $n \in \mathbb{N}$,} \]
and there exists a positive constant $C$ such that
\[ 
\Xi_n \leq C n^{p(1+\beta)} \quad \mbox{for all $n \in \mathbb{N}$.} \]
Note that $p(1+\beta)<1/2$ if $p<1/2$ and $0<\beta<(1-2p)/2p$. As $(1-2p)/2p \geq p/(1-p)$ is equivalent to $p \leq 1/3$, 
\[ \mbox{if $p \leq \dfrac{1}{3}$ then $p(1+\beta)<\dfrac{1}{2}$ for all $\beta \in \left(0,\dfrac{p}{1-p}\right)$.} \]
In view of \eqref{;2.7}, we hope this is useful for comparison with the branching process. 
\end{remark}

The results obtained in this paper are summarized in Table \ref{table1}. 

\begin{table}[ht]
\centering
\begin{tabular}{|c||c|c|} \hline
 Regime & Asymptotic behaviour \\ \hline \hline
$-1<\beta<0$ & $P(\mbox{$\Xi_n\geq C n^{-\beta}$ for all $n$})=1$. \\
& $P\left(
\mbox{$M_{\infty}>0$, $\Xi_n \sim C(p,\beta) M_{\infty} n^{p (\beta+1)-\beta}$ as $n \to \infty$}
\right)>0$. 
\\ \hline
$\beta=0$ & $P\left(\mbox{$M_{\infty}>0$, $\Xi_n \sim  C(p,0) M_{\infty} n^p$ as $n \to \infty$}\right)=1$. \\ \hline
$0<\beta<\dfrac{p}{1-p}$ &  $P(\Xi_{\infty} < + \infty)>0$, \\
& $P\left(\mbox{$M_{\infty}>0$, $\Xi_n \sim   C(p,\beta) M_{\infty} n^{p (\beta+1)-\beta}$ as $n \to \infty$}\right)>0$. 
\\ \hline
$\beta \geq \dfrac{p}{1-p}$ & $P(\Xi_{\infty} < + \infty)=1$. \\ \hline
\end{tabular}
\medskip
\caption{Summary of the results}\label{table1}
\end{table}

\end{document}